\documentclass[11pt]{article}

\usepackage{amsmath}
\usepackage{amssymb}
\usepackage{amsthm}
\usepackage{adjustbox}
\usepackage{color}
\usepackage{verbatim}
\usepackage{epsfig}
\usepackage{latexsym}
\usepackage{array}
\usepackage{graphicx}
\usepackage{hyperref}

\usepackage{a4wide}

\usepackage{authblk}

\DeclareMathOperator{\diag}{diag}

\DeclareMathOperator{\rank}{rank}
\DeclareMathOperator{\lcm}{lcm}

\def\mxl{\left[ \begin{array}}  \def\mxr{\end{array} \right]}
\def\detl{\left| \begin{array}}  \def\detr{\end{array} \right|}
\def\bmx{\mxl{cccc}} \def\emx{\mxr}

\newcommand{\mxx}[2]{\left[ \begin{array}{#1}#2\end{array}\right]} % has 2 commands- can't use \def

\def\bc{\begin{center}}          \def\ec{\end{center}}
\def\ben{\begin{enumerate}}      \def\een{\end{enumerate}}
\def\beq{\begin{equation}}       \def\eequ{\end{equation}}
\def\beqn{\begin{eqnarray*}} \def\eeqn{\end{eqnarray*}} %Pedro
\def\eqs{\vspace{-10pt}\begin{equation}} \def\eqe{\vspace{-5pt}\end{equation}}
\def\bal{\begin{align}}           \def\eal{\end{align}}
\def\ben{\begin{enumerate}}       \def\een{\end{enumerate}}
\def\bit{\begin{itemize}}           \def\eit{\end{itemize}}
\def\btabb{\begin{tabbing}}         \def\etabb{\end{tabbing}}
\def\btab{\begin{tabular}}          \def\etab{\end{tabular}}
\def\barr{\begin{array}}           \def\earr{\end{array}}
\def\earrb{\end{array} \right\}}

\def\xx{\times}

\def\lb{\lambda}

\def\ga{\alpha}

    \def\io{^{-1}}

\def\it{\emph}  
\def\bf{\textbf}  %\def\bf{\mathbf} \def\bo{\textbf}

 \def\x{\bf{x}}  \def\y{\bf{y}}  
\def\z{\bf{z}} %\def\c{\bf{c}} 

    \def\w{\bf{w}}
\def\0{\bf{0}}    

\font\Bbb=msbm10 %blackboard-bold
\def\A{\ifmmode{\mbox{\Bbb A}}\else{{\Bbb A}}\fi}
\def\B{\ifmmode{\mbox{\Bbb B}}\else{{\Bbb B}}\fi}
\def\R{\ifmmode{\mbox{\Bbb R}}\else{{\Bbb R}}\fi}
\def\C{\ifmmode{\mbox{\Bbb C}}\else{{\Bbb C}}\fi}

\def\F{\mathbb{F}}
\def\G{\ifmmode{\mbox{\Bbb G}}\else{{\Bbb G}}\fi}
\def\N{\ifmmode{\mbox{\Bbb N}}\else{{\Bbb N}}\fi}
\def\Z{\ifmmode{\mbox{\Bbb Z}}\else{{\Bbb Z}}\fi}
\def\K{\ifmmode{\mbox{\Bbb K}}\else{{\Bbb K}}\fi}
\def\L{\ifmmode{\mbox{\Bbb L}}\else{{\Bbb L}}\fi}
\def\D{\ifmmode{\mbox{\Bbb D}}\else{{\Bbb D}}\fi}
\def\T{\ifmmode{\mbox{\Bbb T}}\else{{\Bbb T}}\fi}
\def\I{\ifmmode{\mbox{\Bbb I}}\else{{\Bbb I}}\fi}
\def\E{\ifmmode{\mbox{\Bbb E}}\else{{\Bbb E}}\fi}
\def\K{\ifmmode{\mbox{\Bbb K}}\else{{\Bbb K}}\fi}
\def\Q{\ifmmode{\mbox{\Bbb Q}}\else{{\Bbb Q}}\fi}
\def\P{\ifmmode{\mbox{\Bbb P}}\else{{\Bbb P}}\fi}

\def\cG{\ifmmode{\mathcal{G}}\else{{$\mathcal{G}$}}\fi} 
\def\cA{\ifmmode{\mathcal{A}}\else{{$\mathcal{A}$}}\fi}

\def\op{ \ifmmode{\oplus} \else{\leavevmode\hbox{$\oplus$}\fi } }
\def\ot{\ifmmode{\otimes}\else{$\!\!\otimes$}\fi}
\def\od{ \ifmmode{\odot} \else{\leavevmode\hbox{$\odot$}\fi } }
\def\adots{\mathinner{\raise 1pt\hbox{.}\mkern1mu\raise4 pt\hbox{.}\mkern2mu
  \mkern1mu\raise7 pt\vbox{\kern7 pt\hbox{.}}\mkern2mu}}

\def\bdefn{\begin{defn}} \def\edefn{\end{defn}}
\def\bex{\begin{ex}} \def\eex{\end{ex}}
\def\beg{\begin{eg}} \def\eeg{\end{eg}}
\def\blem{\begin{lem}} \def\elem{\end{lem}}
\def\bth{\begin{thm}} \def\eth{\end{thm}}
\def\bprop{\begin{prop}} \def\eprop{\end{prop}}
\def\bcor{\begin{cor}} \def\ecor{\end{cor}}

\newtheorem{thm}{Theorem}[section]  
\newtheorem{lem}{Lemma}[section]
\newtheorem{cor}{Corollary}[section]  \newtheorem{prop}{Proposition}[section]
\newtheorem{eg}{Example}[section]    
\newtheorem{ex}{Exercise}[section]
\newtheorem{defn}{Definition}[section]

\def\ball{\begin{align}}  \def\eall{\end{align}}

\def\bbin{\left( \begin{array}{c}} \def\ebin{\end{array} \right)} %Pedro

\title{On the  index of an anti-triangular block matrix} 

\author[1]{Faustino Maciala\thanks{fausmacialamath@hotmail.com}}
\author[2]{Xavier Mary\thanks{xavier.mary@parisnanterre.fr}}
\author[3]{C. Mendes Ara\'ujo\thanks{clmendes@math.uminho.pt}}
\author[3]{Pedro Patr\'{\i}cio \thanks{pedro@math.uminho.pt}} 

\affil[1]{CMAT -- Centre of Mathematics, Universidade do Minho, 4710-057 Braga, Portugal. Departamento de Ci\^encias da Natureza e Ci\^encias Exatas do Instituto Superior de Ci\^encias da Educa\c c\~ao de Cabinda, Angola.}

\affil[2]{ Laboratoire Modal'X, Universit\'e Paris Nanterre, 200 avenue de la r\'epublique, 92000 Nanterre, France. }

\affil[3]{CMAT -- Centre of Mathematics and Department of Mathematics, Universidade do Minho, 4710-057 Braga, Portugal}

\date{ }
\begin{document}
\maketitle

%{\footnotesize  ORCID: 0000-0002-6466-4426  \,\,   ORCID: 0000-0002-3579-0749 \, \,  ORCID: 0000-0001-9776-0919}}

\begin{abstract}
The  index of a matrix is a fundamental invariant in the analysis of singular matrices and their generalized inverses.
While sharp results are available for block triangular matrices, the corresponding theory for anti-triangular block matrices is less developed. In this paper, we study matrices of the form
\[
M=\left[\begin{array}{cc} A & B \\ C & 0 \end{array}\right],
\]
under algebraic constraints on the blocks.

Building on additive decompositions involving von Neumann inverses, we relate the  index of $M$ to invariance properties of the index and minimal polynomial of expressions of the form $A^{2}A^{-}+I-AA^{-}$, where $A^-$ is any von Neumann inverse of $A$. This connection provides an effective mechanism to control the index of $M$ through suitable factorizations and associated block products.

As a consequence, we derive explicit lower and upper bounds for the index of  $M$ in terms of the index of  $A$ and of $BC$, and characterize situations in which these bounds are attained. Under additional annihilation or orthogonality conditions on the blocks, we obtain closed-form representations for the Drazin inverse of $M$. Applications to adjacency matrices of directed graphs illustrate the sharpness of the bounds and the applicability of the results to structured matrices arising in graph-theoretic settings.
\end{abstract}

%\begin{abstract}
%As far as the authors' knowledge goes, there is no closed general formula for the Drazin inverse of a anti-triangular block matrix, that is, a block matrix with zero $(2,2)$ block. In this paper, we address the index of these block matrices with constraints on the blocks. Finally, we consider matrices related to some types of (weighted) digraphs.
%\end{abstract}

\



\section{Introduction}

The theory of generalized inverses of matrices has been a cornerstone of linear algebra for several decades, with applications ranging from differential equations and control theory to Markov chains. Among these inverses, the Drazin inverse occupies a particularly prominent role in the analysis of singular or non-diagonalizable matrices, especially those arising from dynamical systems and stochastic processes. First introduced by M.P. Drazin in 1958 in semigroups  \cite{Drazin1958}, the Drazin inverse $A^D$ of a square matrix $A$ over an arbitrary field is the unique matrix satisfying
$$
A^{k+1}A^D = A^k, \quad A^DAA^D = A^D, \quad \text{and} \quad AA^D = A^DA,
$$
where $k$ is the  Drazin index of $A$, denoted by $i(A)$.  The Drazin index equals the index of the matrix; that is, if $\psi_A(\lambda)$ denotes the minimal polynomial of the matrix $A$, then $\psi_A(\lambda) = \lambda^{i(A)} f(\lambda)$, where $f(0) \neq 0$. It is known that the index of $A$ is the smallest  nonnegative integer $k$ for which $\operatorname{rank}(A^k) = \operatorname{rank}(A^{k+1})$. Equivalently, it is the smallest  nonnegative integer $k$ for which $\ker (A^{k}) = \ker(A^{k+1})$, alternatively  $R(A^k)=R(A^{k+1})$, where $\ker(\cdot)$ and $R(\cdot)$ denote, respectively,  the kernel and the range of a matrix. The group inverse of a matrix $A$, denoted by $A^\#$, is a special case of the Drazin inverse whose index is at most $1$.

    The standard notation $A\{1\}$ is used for the set of von Neumann inverses of $A$, that is, the set of solutions to the matrix equation $AXA=A$. A particular von Neumann inverse will be denoted by $A^-$.  We may need to consider different choices of von Neumann inverses of $A$, and for that purpose we denote by $A^=\in A\{1\}$.  A reflexive inverse of $A$, denoted by $A^+$, is a common solution to the matrix equations $AXA=A, XAX=X$. It is easy to check that $A^- A A^=$ serves as a reflexive inverse of $A$, for any choice of $A^-, A^=\in A\{1\}$.  
For further definitions and results concerning generalized inverses of matrices, the reader is referred to \cite{Adi, CampbellMeyer}.

 As usual, the vector space of $m\xx n$ matrices of a general field $\F$ is denoted by $\F^{m\xx n}$.

An important topic in the algebraic theory of generalized invertibility, namely von Neumann, group and Drazin inverses, is to provide a closed formula for these inverses for block matrices.  In recent years, 
considerable progress has been achieved in representing the Drazin inverse of block matrices and block 
operator matrices. The extant literature contains several recent references examining Drazin invertibility of an anti-triangular matrix, such as \cite{Bu2012, Bu2011, Nieves, CatralDrazin, Dopazo, Liu, Drazin220, Zhao, Zhang2020, Zhang2019}.   However, relatively little attention has been given to the explicit characterization 
of the Drazin index. It is worth noting that for a large square matrix $A$, determining the Drazin index 
$i(A)$ in terms of $\operatorname{rank}(A^k)$ ($k \in \mathbb{N}$) can be quite challenging, 
as these ranks are often difficult to compute. Consequently, various techniques involving partitioned 
matrices are commonly employed to address this issue. In particular, in their seminal paper, and independently, Hartwig and Shoaf \cite{HartwigShoaf} and Meyer and Rose \cite{Meyer} addressed block triangular matrices, and in particular showed that if $M$ is a block triangular matrix with diagonal blocks $A$ and $B$, then $\max \{i(A), i(B)\}\le i(M)\le i(A)+i(B)$. This was later addressed by Bru et al. \cite{Bru} by characterizing $M$ for which its index takes values in between the lower and upper bound, and revisited by Xu et al. \cite{Xu} in the computation of the explicit Drazin indices of certain $2 \times 2$ operator matrices.

The foundation of the technique for studying the problem essentially rests upon some form of additive matrix decomposition, featuring some type of one-sided orthogonality, which at some point allows for the application of Cline's lemma. However, the repeated application of this technique does not allow for effective control over the index of the matrix, since new inequalities arise at each step where Cline's lemma  is applied. Although we also apply Cline's lemma at an early stage, our approach uses other techniques that allow us to associate the matrix with another one with a lower  index.

This work starts with the presentation of a series of preparatory results, we then relate the %Drazin
 index and the minimal polynomial of some special sums, we study the index of an anti-triangular block matrix subject to block constraints, and we conclude with some applications to matrices associated to certain types of digraphs.
 
\section{Lemmata}  

In this section we collect a number of auxiliary results which will be used in the upcoming sections.

\begin{lem}\label{indexAB}
 Given matrices $A$ and $B$ such that both products $AB$ and $BA$ are defined, one of the following identities holds:
\[
\psi_{AB}(\lambda)=\psi_{BA}(\lambda), \qquad
\psi_{AB}(\lambda)=\lambda\,\psi_{BA}(\lambda), \qquad
\text{or} \qquad
\psi_{BA}(\lambda)=\lambda\,\psi_{AB}(\lambda).
\]
 %   Given matrices $A$ and $B$ of conformal sizes, we have {\color{blue} either  $ \psi_{AB}(\lb)= \psi_{BA}(\lb)$ or $ \psi_{AB}(\lb)=\lb \,\psi_{BA}(\lb)$ or  $\psi_{BA}(\lb)=\lb \,\psi_{AB}(\lb)$. }
%    \[
 %  \psi_{AB}(\lb)=\lb^{0,\pm 1}\,\psi_{BA}(\lb).
%    \]
\end{lem}
\proof 
From $(AB)^{n+1}=A(BA)^n B$ it follows that
\[
B\,\psi_{AB}(AB)\, A \;=\; BA\,\psi_{AB}(BA),
\]
hence $\psi_{BA} (\lb)\mid \lb\,\psi_{AB}(\lb)$. 
Similarly, $\psi_{AB} (\lb)\mid \lb\,\psi_{BA}(\lb)$.   The result follows.
%Therefore $\psi_{AB}(\lb)=\lb^{0,\pm 1}\,\psi_{BA}(\lb)$.
\endproof

\noindent\emph{Example.}
Consider the matrices over a field:
\[
A= \left[\begin{array}{rr|rr}
0 & 1 & 0 & 0 \\
0 & 0 & 0 & 0 \\
\hline
0 & 0 & 0 & 1 \\
0 & 0 & 0 & 0
\end{array}\right],\quad
B=\left[\begin{array}{rrr|r}
0 & 1 & 0 & 0 \\
0 & 0 & 1 & 0 \\
0 & 0 & 0 & 0 \\
\hline
0 & 0 & 0 & 1
\end{array}\right],\quad
C= \left[\begin{array}{rrr|r}
0 & 1 & 0 & 0 \\
0 & 0 & 1 & 0 \\
0 & 0 & 0 & 0 \\
\hline
0 & 0 & 0 & 0
\end{array}\right].
\]
Direct computation shows that $BA$, $AC$ and $CA$ are nilpotent of index $2$, whereas $AB$ is nilpotent of index $3$. 
Consequently,
\[
\lb^2 = \psi_{BA}(\lb)= \psi_{AC}(\lb)=\psi_{CA}(\lb)=\lb^{-1}\psi_{AB}(\lb)=\lb^{-1}\cdot \lb^3.
\]

\begin{lem}[Cline's Lemma]\label{Cline}
    Given matrices $A$ and $B$ of conformal sizes, we have
    \[
    (AB)^D = A((BA)^D)^2 B, \quad \text{and} \quad |\,i(AB)-i(BA)\,|\le 1.
    \]
\end{lem}
\proof See \cite[Theorem 7.8.4]{CampbellMeyer}. The index inequality follows from Lemma \ref{indexAB}. \endproof

\begin{lem} 
 Given matrices $A$ and $B$ of conformal sizes, if $AB=BA$ then
\[
AB^D=B^DA \quad \text{and} \quad A^D B^D=B^D A^D.
\]
\end{lem}
\proof See \cite[Theorem 7.8.4]{CampbellMeyer}. \endproof

\begin{lem}\label{Apower}
    Given any von Neumann inverse $A^-$ of a square matrix $A$ and a positive integer $\ell$,
    \[
    (A^2A^-)^\ell = A^{\ell +1}A^-.
    \]
\end{lem}
\proof
The proof proceeds by induction. For $\ell=1$ the claim is immediate. 
Assuming it holds for $\ell$, we compute
\[
(A^2A^-)^{\ell +1} =A^2A^- (A^2A^-)^{\ell} 
= A^2A^- A^{\ell +1}A^- =A^{\ell +2}A^-,
\]
which establishes the result.
\endproof

\begin{lem}\label{Npower}
    $N^{k+1}=0\ne N^{k}$ if and only if $(N^2N^-)^k =0 \ne (N^2 N^-)^{k-1}$ for one (and hence all) choices of von Neumann inverse $N^-$ of a square matrix $N$.
\end{lem}
\proof 
For the `if' part, observe that $(N^2N^-)^k=0$ is equivalent to $N^{k+1}N^- = N^{k+1}=0$. 
If $N^k=0$ then $N^kN^-=(N^2N^-)^{k-1}=0$, a contradiction.  

Conversely, $N^{k+1}=0$ is equivalent to $N^{k+1}N^-=0$, which implies $(N^2N^-)^k=0$. 
If $(N^2N^-)^{k-1}=0$, then $N^k=0$, again a contradiction. 
\endproof

\begin{lem}\label{orthsum} Given matrices $X$ and $Y$ of conformal sizes such that  $X+Y$ is singular, and  $XY=YX=0$ then
    \[
    (X+Y)^D=X^D+Y^D 
%    \quad \text{and} \quad i(X+Y)=\max\{i(X),i(Y)\}.
    \]
Furthermore,
\ben
\item If $i(X)\ne i(Y)$ or $\min\{i(X), i(Y)\}=1$ then $i(X+Y)=\max\{i(X),i(Y)\}$;
\item If $i(X)=i(Y)\ne 1$ then $i(X)-1\le i(X+Y) \le i(X)$.
\een
\end{lem}
 
 \proof  The proposed expression satisfies the three defining equations of the Drazin inverse. Hence, by uniqueness, $(X+Y)^D = X^D+Y^D$.

Set $r=i(X+Y)$ and $k=\max\{i(X),i(Y)\}$.

Since $XY=YX=0$, we have $(X+Y)^\ell =X^\ell + Y^\ell$ for every integer $\ell\geq 1$. Moreover, the relations $X^DY=YX^D=Y^DX=XY^D=0$ yield  $(X+Y)^k(X^D+Y^D)(X+Y)=(X+Y)^k$. Therefore, $r\le k$. 
 
Furthermore, for every integer $m\geq 0$,
$$
X(X+Y)^m=X^{m+1}
\quad \text{and} \quad
Y(X+Y)^m=Y^{m+1}.
$$
Notice also that both $X$ and $Y$ commute with $X+Y$.

Let $v\in\ker X^{r+2}$. Since $X$ commutes with $X+Y$, it follows that
$$
(X+Y)^{r+1}(Xv)=X(X+Y)^{r+1}v= X^{r+2}v=0.$$
Thus, $Xv\in\ker (X+Y)^{r+1}$. Since $i(X+Y)=r$, we have $\ker (X+Y)^{r+1}=\ker (X+Y)^r$.
Consequently, $(X+Y)^r(Xv)=0$. Using the identities above once again, we obtain
$$
X^{r+1}v=0.
$$
Hence,
$$
\ker X^{r+2}\subseteq\ker X^{r+1}.
$$
The reverse inclusion is immediate, and therefore
$$
\ker X^{r+2}=\ker X^{r+1}.
$$
It follows that $i(X)\leq r+1$.

By symmetry, applying the same argument to $Y$ gives
$i(Y)\leq r+1$. Consequently, $k\leq r+1$, or, equivalently, $r\geq k-1$.

Combining this inequality with the previously established upper bound, we obtain
$$
k-1
\leq r
\leq k.
$$

Suppose first that $i(X)=i(Y)=1$. Then the inequality $k-1\leq r\leq k$ gives $0\leq r\leq 1$. Since $X+Y$ is singular by hypothesis, its Drazin index cannot be zero. Therefore, $r=1=k$.

Suppose now, without loss of generality, that $k=i(X)\geq i(Y)$ and $i(X)>1$. If $r\neq k$  then $r=k-1=i(X)-1$. Since $(X+Y)^r(X^D+Y^D)(X+Y)=(X+Y)^r$, we obtain 
$$X^{i(X)-1}X^DX+Y^{i(X)-1}Y^DY=X^{i(X)-1}+Y^{i(X)-1}.$$ Assume, in addition, that $i(Y)<i(X)$. Then $Y^{i(X)-1}Y^DY=Y^{i(X)-1}$, and hence the preceding equality implies $X^{i(X)-1}X^DX=X^{i(X)-1}$, which contradicts the definition of $i(X)$. Therefore, $i(Y) <i(X)$ does not hold, which means $i(X)=i(Y)$. This proves that, whenever $i(X)\neq i(Y)$, $r=k$.
\endproof

\begin{lem}\label{lemma2-7}
    Given any von Neumann inverse $A^-$ of a square matrix $A$,
\begin{enumerate}
    \item  $(A^2A^- + I -AA^-)^\ell = A^{\ell +1}A^-+I-AA^-$;
    \item If $A$ is not group invertible, then $i(A^2A^-+I-AA^-)=i(A^2A^-)$, and
    \[
    (A^2A^-+I-AA^-)^D = (A^2A^-)^D + I-AA^-.
    \]
\end{enumerate}     
\end{lem}
\proof (2) follows from Lemma \ref{orthsum}(1), taking into account that $i(I-AA^-)=1\le i(A^2A^-)$. \endproof 

\begin{lem}
\label{schurcomp}
Let
\[
M = 
\mxl{cc}
A & B \\
C & \D
\mxr
\]
be a block matrix with $A$ invertible, and its associated Schur complement $Z= D - C A^{-1}B$.
\begin{enumerate}
\item Given a von Neumann inverse $Z^-$ of $Z$, then
$$M^-   = 
\mxl{cc}
I & -A^{-1} B \\
0 & I
\mxr
\mxl{cc}
A^{-1} & 0 \\
0 & Z^{-}
\mxr
\mxl{cc}
I & 0 \\
-C A^{-1} & I
\mxr$$ is a von Neumann inverse of $M$. 

\item $M$ is invertible if and only if $Z$ is invertible, in which case
\[
M^{-1} = 
\mxl{cc}
I & -A^{-1} B \\
0 & I
\mxr
\mxl{cc}
A^{-1} & 0 \\
0 & Z^{-1}
\mxr
\mxl{cc}
I & 0 \\
-C A^{-1} & I
\mxr.
\]
\end{enumerate}
\end{lem}
\proof The proof follows by considering the factorization $$
M = 
\mxl{cc}
I & 0 \\
C A^{-1} & I
\mxr
\mxl{cc}
A & 0 \\
0 & D - C A^{-1}B
\mxr
\mxl{cc}
I & A^{-1} B \\
0 & I
\mxr.$$ \endproof

\begin{lem}
Let $W$ be a square matrix and $ W^{-} \in W{\{1\}}$. Then, for any positive integer $ n $,
\begin{itemize}
    \item $(W^D W W^{-})^{n+1} = (W^D)^{n} W^{-};$
    \item $W^D W W^{-} W^D = (W^D)^2.$
\end{itemize}
\end{lem}

\begin{lem}\label{Y_from_W}
Let $Y= \mxl{cc} 0 & WW^{-} \\  W & 0 \mxr$ where $W$ is a singular matrix and $W^-\in W\{1\}$ . Then $i(Y)=2i(W)-1$ and $$Y^D=\mxl{cc} 0 & W^DWW^{-} \\ WW^D & 0 \mxr.$$
\end{lem}
\proof

Let  $i(W)=k\geq 1$. In order to show  $ i(Y)=2k-1$,  we claim that $Y^{2l}=\mxl{cc} W^l & 0 \\   0 & W^{l+1}W^{-} \mxr$, which we prove by induction.

For $l=1$ the equality holds since $ Y^{2}=\mxl{cc} W & 0 \\  0 & W^2W^{-} \mxr$. For the inductive step, 
    \begin{align*}
        Y^{2(l+1)} &=Y^{2l}Y^2 \\
        &= \mxl{cc} W^l & 0 \\ 0 & W^{l+1}W^{-} \mxr \mxl{cc} W & 0 \\  0 & W^2W^{-} \mxr \\
        &= \mxl{cc} W^{l+1} & 0 \\  0 & W^{l+1}WW^2W^{-} \mxr \\
        &= \mxl{cc} W^{l+1} & 0 \\  0 & W^{l}WW^{-}WWW^{-} \mxr \\
        &= \mxl{cc} W^{l+1} & 0 \\  0 & W^{l+2}W^{-} \mxr. 
    \end{align*}
    
 Furthermore,
\begin{align*}
    Y^{2l+1} = YY^{2l} &= \mxl{cc} 0 & WW^{-} \\ W & 0 \mxr \mxl{c c} W^l & 0 \\  0 & W^{l+1}W \mxr \\
    &= \mxl{cc} 0 & WW^{-}W^{l+1}W^{-} \\ W^{l+1} & 0 \mxr = \mxl{cc} 0 & WW^{-}WW^{l}W^{-} \\ W^{l+1} & 0 \mxr \\
    &= \mxl{cc} 0 & W^{l+1}W^{-} \\ W^{l+1} & 0 \mxr
\end{align*}
 and $i(Y)\leq 2k-1$. 
 
 Suppose now $i(Y) < 2k-1.$ Then 
$CS(Y^{2k-2})=CS(Y^{2k-1})$ with $Y^{2k-2}=Y^{2(k-1)}=\mxl{cc} W^{k-1} & 0 \\ 0 & W^kW^{-}\mxr$ and $Y^{2k-1}= \mxl{cc} 0 & W^kW^{-} \\ W^k & 0 \mxr.$ 
As $CS\left( \mxl{c} 0 \\ W^k \mxr \right)=CS\left( \mxl{c} 0 \\ W^kW^{-} \mxr \right)$, then $CS\left( \mxl{c} W^{k-1} \\ 0\mxr \right)=CS\left( \mxl{c} W^{k}W^{-} \\ 0\mxr \right)$, which gives  $CS\left( W^{k-1} \right)=CS\left( \ W^k  \right)$,  contradicting  $i(W)=k$. 

So, $i(Y)=2k-1=2i(W)-1.$ 

We are left to show that 
$Y^D=\mxl{cc} 0 & W^DWW^{-} \\ WW^D & 0 \mxr=Z$. We now check $Z$ satisfies Drazin's equations. 

\begin{itemize}
    \item[(a).] 
    \begin{eqnarray*}YZ&=& \mxl{cc} 0 & WW^{-} \\ W & 0 \mxr \mxl{cc} 0 & W^DWW^{-} \\ WW^D & 0 \mxr=\mxl{cc} WW^{-}WW^D & 0 \\ 0 & WW^DWW^{-} \mxr\\
    & =& \mxl{cc} W^DWW^{-}W & 0 \\ 0 & WW^DWW^{-} \mxr = \mxl{cc} 0 & W^DWW^{-} \\ WW^D & 0 \mxr \mxl{cc} 0 & WW^{-} \\ W & 0 \mxr = ZY.
        \end{eqnarray*}

    \item[(b).] \begin{eqnarray*}ZYZ &= &Z(ZY)=\mxl{cc} 0 & W^DWW^{-} \\ WW^D & 0 \mxr\mxl{cc} W^DW & 0 \\ 0 & WW^DWW^{-} \mxr \\
    &=& \mxl{cc} 0 & W^DWW^{-}WW^DWW^{-} \\ WW^DW^DW & 0 \mxr\\
    &=& \mxl{cc} 0 & W^DWW^DWW^{-} \\ W^DWW^DW & 0 \mxr \\
    &=& \mxl{cc} 0 & W^DWW^{-} \\ W^DW & 0 \mxr =Z.
    \end{eqnarray*}
    \item[(c).] We can take $i(W)=k,$ we still need to verify that $Y^{(2k-1)+1}Z=Y^{2k-1}$, since $i(Y)=2k-1$, i.e., $Y^{2k}Z=Y^{2k-1}.$
\end{itemize}

Since $W^kW^DWW^{-}=W^kWW^DW^{-}=W^{k+1}W^DW^{-}$ and $W^{k+1}W^{-}WW^D=W^kWW^{-}WW^D=W^{K+1}W^D$ and with $W^{k+1}W^D=W^k$, we have 
\begin{eqnarray*}
Y^{2k}Z&=& \mxl{cc} W^k & 0 \\ 0 & W^{k+1}W^{-}\mxr
\mxl{cc} 0 & W^DWW^{-} \\ WW^D & 0\mxr  \\
&=& \mxl{cc} 0 & W^kW^DWW^{-} \\ W^{k+1}W^{-}WW^D & 0\mxr
= \mxl{cc} 0 & W^{k+1}W^DW^{-} \\ W^{k+1}W^D & 0\mxr \\
&=& \mxl{cc} 0 & W^{k}W^{-} \\ W^{k} & 0\mxr = Y^{2k-1}=Y^{2(k-1)+1}.
\end{eqnarray*}

From (a), (b) and (c) we can conclude, in fact, that $Y^D=Z.$

\endproof

\begin{lem}
\label{YY^Dpowers}
Let \( Y = \mxl{cc} 0 & W W^{-} \\ W & 0 \mxr \), where \( W \) is a square matrix with   \( W^{-} \in W{\{1\}} \). Then, for any integer \( n \geq 1 \),  

\ben
    \item $Y^n=
\begin{cases}
\mxl{cc} W^{\frac{n}{2}} & 0 \\ 0 & W^{\frac{n}{2}+1 } W^- \mxr, & n \text{ is even} \\ 
\\
\mxl{cc} 0 & W^{\frac{n+1}{2}}W^{-} \\ W^{\frac{n+1}{2}} & 0 \mxr,  & n \text{ is odd}
\end{cases}$

\item $(Y^D)^n=
\begin{cases}
\mxl{cc} (W^D)^{\frac{n}{2}} & 0 \\ 0 & (W^D)^{\frac{n}{2}}WW^{-} \mxr, & n \text{ is even} \\ 
\\
\mxl{cc} 0 & (W^D)^{{\frac{n+1}{2}}}WW^{-} \\ (W^D)^{\frac{n+1}{2}} & 0 \mxr,  & n \text{ is odd}
\end{cases}$
\een
\end{lem}

\section{The Drazin index and minimal polynomials of special sums}

We now present results concerning the Drazin inverse, and in particular the connection between Drazin indices and  von Neumann invertibility. We further explore these relations by considering minimal polynomials.

\begin{prop}\label{jacobson}
Given matrices $A$ and $B$ such that both products $AB$ and $BA$ are defined, one of the following identities holds:
\[
\psi_{I-AB}(\lambda)=\psi_{I-BA}(\lambda), \quad
\psi_{I-AB}(\lambda)=(\lambda-1)\psi_{I-BA}(\lambda), \quad
\text{or} \quad
\psi_{I-BA}(\lambda)=(\lambda-1)\psi_{I-AB}(\lambda).
\]
Moreover,
\[
i(I-AB)=i(I-BA).
\]
%Given matrices $A$ and $B$ of conformal sizes, 
%\[
%\psi_{I-AB}(\lb) = (\lb-1)^{0, \pm1} \psi_{I-BA}(\lb) \text{ and }i(I-AB)=i(I-BA).
%\]
\end{prop}
\proof 
Set $X=AB$, $Y=BA$, $K=I-X$ and $W=I-Y$. Consider the factorization $\psi_K (\lb) = \lb^k f(\lb)$, where $gcd(\lb, f(\lb))=1$.
%Then
%\[
%\lb^k f(\lb)= \psi_K (\lb)=\psi_X(\lb-1)=(\lb-1)^{0,\pm 1} \psi_Y(\lb-1).
%\]
Then one of the following identities holds:
\[
\lambda^k f(\lambda)=\psi_Y(\lambda-1), \quad
\lambda^k f(\lambda)=(\lambda-1)\psi_Y(\lambda-1), \quad
\text{or} \quad
\psi_Y(\lambda-1)=(\lambda-1)\lambda^k f(\lambda).
\]
It follows that $\psi_Y(\lb-1)=\psi_W(\lb)=\lb^w g(\lb)$, where $gcd(\lb, g(\lb))=1$.  
Consequently, one of the following identities holds:
\[
\lambda^k f(\lambda)=\lambda^w g(\lambda), \quad
\lambda^k f(\lambda)=(\lambda-1)\lambda^w g(\lambda), \quad
\text{or} \quad
\lambda^w g(\lambda)=(\lambda-1)\lambda^k f(\lambda).
\]
%\[
%\lb^k f(\lb)=(\lb-1)^{0,\pm 1}\,\lb^w g(\lb).
%\]
Therefore $k=w$ and the Drazin indices of $I-AB$ and $I-BA$ are equal.
\endproof

As an example, consider $A= \left[\begin{array}{rrr|rr}
0 & 1 & 0 & 0 & 0 \\
0 & 0 & 1 & 0 & 0 \\
0 & 0 & 0 & 0 & 0 \\
\hline
 0 & 0 & 0 & -1 & 0 \\
0 & 0 & 0 & -2 & 2
\end{array}\right], B = \left[\begin{array}{rr|rr|r}
0 & 1 & 0 & 0 & 0 \\
0 & 0 & 0 & 0 & 0 \\
\hline
 0 & 0 & 0 & 0 & 0 \\
0 & 0 & -1 & -1 & 0 \\
\hline
 0 & 0 & 0 & 0 & -3
\end{array}\right]$.  Then $AB= \left[\begin{array}{rrrrr}
0 & 0 & 0 & 0 & 0 \\
0 & 0 & 0 & 0 & 0 \\
0 & 0 & 0 & 0 & 0 \\
0 & 0 & 1 & 1 & 0 \\
0 & 0 & 2 & 2 & -6
\end{array}\right], BA= \left[\begin{array}{rrrrr}
0 & 0 & 1 & 0 & 0 \\
0 & 0 & 0 & 0 & 0 \\
0 & 0 & 0 & 0 & 0 \\
0 & 0 & 0 & 1 & 0 \\
0 & 0 & 0 & 6 & -6
\end{array}\right]$, $\psi_{I-AB}(\lb) = (\lb - 7) \cdot (\lb - 1) \cdot \lb$ and $\psi_{I-BA}(\lb) =(\lb - 7) \cdot \lb \cdot (\lb - 1)^{2}$.  Indeed,  the indices of $I-AB$ and of $ I-BA $ are both equal to $1$ and $\psi_{I-BA}(\lb) = (\lb-1)\psi_{I-AB}(\lb) $.

\begin{prop}
    Given a square singular matrix $A$, there exists a von Neumann inverse $A^-$ of $A$ such that
    \[
    i(A) = i(A^2 A^- + I-AA^-) +1.
    \]
\end{prop}
\proof 
It suffices to show that there exists $A^-$ such that $i(A)=i(A^2A^-)+1$, in view of Lemma \ref{lemma2-7}. 
Consider a core-nilpotent decomposition 
\[
A=U\mxx{cc}{ C & 0\\ 0 & N }U^{-1},
\]
where $C$ is nonsingular and $N$ is nilpotent of index $i(A)$.  
Consider the von Neumann inverse of $A$
\[
A^-=U\mxx{cc}{ C^{-1} & 0\\ 0 & N^- }U^{-1},
\]
where $N^-$ is a von Neumann inverse of $N$.  

Furthermore,
\[
A^2A^-=U\mxx{cc}{ C & 0\\ 0 & N^2 N^- }U^{-1}.
\]
Since the Drazin index is invariant under similarity, and because the nilpotency index of $N$ coincides with its Drazin index, we obtain
\[
i(A^2A^-)=i(N^2N^-)=i(N)-1=i(A)-1.
\]
\endproof

\begin{thm}\label{lem:inv-Aminus}
    Given a square matrix $A$, the Drazin index of $ A^2A^- + I - AA^-$ is invariant under the choice of von Neumann inverse $A^-$ of $A$.  
\end{thm}
\proof 
Let $A^-$ and $A^=$ be two (possibly distinct) von Neumann inverses of $A$. Then
\beqn
i(A^2A^- + I-AA^-) &=& i(I+AA^=(A^2A^--AA^-))\\
&=& i(I+(A^2A^--AA^-)AA^=)\\
&=& i(A^2A^=+I-AA^=).
\eeqn
Hence, the Drazin index is independent of the choice of von Neumann inverse. 
\endproof

\begin{thm} \label{DrazinInv}
Let $A$ be a singular matrix. The following quantities are invariant under the choice of a von Neumann inverse $A^{-} $ of $A$, and all are equal to $i(A)-1$:
\begin{itemize}
    \item $i(A^2A^-+I-AA^-)$,
    \item $i(A+I-AA^-)$,
    \item $i(A^-A^2 +I- A^-A)$,
    \item $i(A+I-A^-A)$.
\end{itemize}
Moreover, $$A^D = \left( \left(A^2A^-+I-AA^-\right)^D\right)^2 A.$$

\end{thm}

\proof 
Observe that 
\[
\begin{aligned}
i(I + AA^-(A - AA^-)) &= i(I + (A - AA^-)AA^-) \\
&= i(I + A(AA^- - A^-)) \\
&= i(I + (AA^- - A^-)A) \\
&= i(I + (A - A^-A)A^-A) \\
&= i(I + A^-A(A - A^-A)).
\end{aligned}
\]
This chain of equalities shows that all four expressions have the same Drazin index, independently of the choice of $A^-$. Since $i(A^2A^- + I - AA^-)=i(A)-1$, the result follows.

In order to obtain the expression for $A^D$, note that $(A^2A^-)^D=(A(AA^-))^D=A (A^D)^2AA^-$ using Lemma \ref{Cline}. Therefore, $\left( (A^2A^-)^D\right)^2= (A^D)^3A^2A^- = A^DA^-$. Then \beqn
\left( \left(A^2A^-+I-AA^-\right)^D\right)^2 A  &=& \left( \left((A^2A^-)^D\right)^2+I-AA^-\right)  A\\
&=& \left( \left(A^2A^- \right)^D\right)^2 A\\
&=& A\left(A^D\right)^3 A \\
&=&A^D
\eeqn
\endproof

\begin{thm}
Let $A$ be a singular matrix.  Then, for every $A^-\in A\{1\}$, 
%$$ \psi_A(\lb) = \lb \psi_{A^2A^-}(\lb) \text{ and } \psi_{A^2A^-+I-AA^-}(\lb)=  \lcm\{\lb\io\psi_A(\lb), \lb-1\}.$$
$$ \psi_{A^2A^-+I-AA^-}(\lb)=  \lcm\{\lb\io\psi_A(\lb), \lb-1\}.$$
If $i(A)>1$ then  $ \psi_A(\lb) = \lb \psi_{A^2A^-}(\lb)$, and $ \psi_A(\lb) = \psi_{A^2A^-}(\lb)$ if $i(A)=1.$

\end{thm}

\proof
Since similar matrices have the same  minimal polynomial, we can consider, without loss of generality, that $A=\bmx C & 0\\ 0 & N\emx$, where $C$ is nonsingular and $N^k=0\ne N^{k-1}$, for some natural $k$. Therefore, $\psi_A(\lb) = \lb^k \psi_C(\lb)$, with $gcd(\lb, \psi_C)=1$, and $k=i(A)$.

%{\color{blue} Firstly we prove $\psi_{A^2A^-}(\lb)$ is invariant under the choice of $A^-\in A\{1\}$. }

 Let $N^-\in N\{1\}$. Since $\psi_N(\lb)=\lb^k$ then, using Lemma \ref{Npower}, we have $\psi_{N^2N^-}(\lb) = \lb^{k-1}$  if $k>1$, and $\psi_{N^2N^-}(\lb) = \lb$ if $k=1$. Taking $A^- = \bmx C\io & 0\\ 0 & N^-\emx$ we have $A^2 A^- = \bmx C & 0\\ 0 & N^2N^-\emx$ which implies that,  if $k>1$, that $\psi_{A^2A^-}(\lb) = \lb^{k-1} \psi_C(\lb)$, which implies, $\psi_A(\lb)=\lb \psi_{A^2A^-}(\lb)$, and $\psi_A(\lb)= \psi_{A^2A^-}(\lb)$ if $k=1$.
Also, $A^2A^-+I-AA^- = \bmx C & 0 \\ 0 & N^2N^-+I-NN^-\emx$, which gives $$\psi_{A^2A^-+I-AA^-}(\lb) = \lcm\{\psi_C(\lb), \psi_{N^2N^-+I-NN^-}(\lb)\}.$$

We claim that $\psi_{N^2N^-+I-NN^-}(\lb) = \lb^{k-1}(\lb-1)$. This is obvious when  $k=1$. Suppose now $k\ge 2$. Using Lemma \ref{Npower}, we have $(N^2N^-+I-NN^-)^{k-1}NN^-=(N^2N^-)^{k-1}NN^-=0$, and also $(N^2N^-)^{k-2}\ne 0$. Note that $(N^2N^-+I-NN^-)^\ell NN^-=(N^2N^-)^\ell$ which is zero if and only if $\ell \ge k-1$. That is, $\lb^\ell(\lb-1)$ is an annihilating polynomial for $N^2 N^-+I-NN^-$ if and only if $\ell\ge k-1$. Therefore, the minimal polynomial is $\lb^{k-1}(\lb-1)$ as desired.

We now prove that the minimal polynomial $\psi_{A^2A^-}(\lb)$ and  $\psi_{A^2A^-+I-AA^-}(\lb)$ are invariant under the choice of $A^-\in A\{1\}$. Let $A^=\in A\{1\}$ arbitrary. From \cite[Corollary 1, p.52]{Adi}, we know there exists $Z=\bmx Z_1 & Z_2\\Z_3 & Z_4\emx$ such that $A^= = A^-+Z-A^-AZAA^-$, and consequently
$$A^2A^= = \bmx C & C^2Z_2(I-NN^-)\\ 0 & N^2N^-+N^2Z_4(I-NN^-)\emx.$$  Therefore, $\psi_{A^2A^=}(\lb) = \lcm\{\psi_C(\lb), \psi_{N^2N^-+N^2Z_4(I-NN^-)}(\lb)\}$ since the (2,2) block is nilpotent and the (1,1) block is invertible. 

We will now prove that $\psi_{N^2N^-+N^2Z_4(I-NN^-)}(\lb) = \psi_{N^2N^-}(\lb)$. If $N=0$, that is, $i(A)=1$, there is nothing left to prove. Suppose now $k\ge 2$. By induction, one can show that 
$$(N^2N^-+N^2Z_4(I-NN^-))^\ell  = N^{\ell+1}N^-+N^{\ell+1}Z_4(I-NN^-).$$ This means $\psi_{N^2N^-}(\lb)$ is a monic annihilating polynomial for $N^2N^-+N^2Z_4(I-NN^-)$. If $\lb^{k-2}$ was to be a monic annihilating polynomial for $N^2N^-+N^2Z_4(I-NN^-)$ then
$$(N^2N^-+N^2Z_4(I-NN^-))^{k-2} = N^{k-1}N^-+N^{k-1}Z_4(I-NN^-)=0$$ would imply, post-multiplying by $N$, that $N^{k-1}=0$ which we assumed to be nonzero. We obtain, therefore, $\psi_{A^2A^=}(\lb) = \lb^{k-1} \psi_C(\lb)$, for any $A^= \in A\{1\}$.

For the invariance of  $\psi_{A^2A^=+I-AA^=}(\lb) $ under the choice of $A^=\in A\{1\}$, 
\iffalse we have 
$$A^2A^= + I-AA^==\bmx C & C^2Z_2(I-NN^-)-CZ_2(I-NN^-)\\
0 & X+Y\emx, $$ 
where $X= N^2N^-+I-NN^-$ and $Y= (N^2-N)Z_4(I-NN^-)$. Note that $Y^2 = 0$, $XY=NY$, $YX=Y$ and $X^\ell = N^{\ell+1}N^-+I-NN^-$. Furthermore,
$$(X+Y)^\ell = N^{\ell+1}N^- + I-NN^- + (N^{\ell+1}-N)Z_4(I-NN^-).$$ 
We now show that $\lb^{k-1}(\lb -1)$ is an annihilating polynomial for $X+Y$. Indeed, and since $N^k=0$,
\begin{eqnarray*}
(X+Y)^{k-1}(X+Y-I) &=& (N^k+I-NN^- + (N^k-N)Z_4(I-NN^-))\times \\
&\times& (N^2N^--NN^-+(N^2-N)Z_4(I-NN^-))\\
&=& (I-NZ_4)(I-NN^-)N(NN^-N^-+(N-I)Z_4(I-NN^-))\\
&=& 0
\end{eqnarray*}

Suppose now $X+Y$ is nilpotent, that is, there exists $\ell$ such that $(X+Y)^\ell = 0$. If that was the case, and since we can write $$ 0 = (X+Y)^\ell = I-N(I-N^\ell)(N^-+Z_4(I-NN^-))$$
then 
$$ N((I-N^\ell)(N^-+Z_4(I-NN^-)) =I$$ and $N$ would be invertible, which cannot be. Therefore $\psi_{X+Y}(\lb)\not| \, \lb^\ell$ for any $\ell $.

We are left to show $\lb^{k-2}(\lb-1)$ is not as annihilating polynomial for $X+Y$. If that was the case,
\begin{eqnarray*}
0 &=& (X+Y)^{k-2}(X+Y-I)\\
&=& -N^{k-1}(N^-+Z_4(I-NN^-))
\end{eqnarray*}
which would lead to $N^{k-1}N^- = -N^{k-1}Z_4(I-NN^-)$. Post-multiplying by $N$ gives $N^{k-1}=0$ which cannot be.
\fi
write $$\psi_{A^2A^-+I-AA^-}(\lb)=\lb^n+\ga_{n-1}\lb^{n-1}+\dots + \ga_1\lb+\ga_0.$$ Since $\left(A^2A^-+I-AA^-\right)^\ell = A^{\ell+1}A^-+I-AA^-$ from Lemma \ref{lemma2-7}, we have the equality
\begin{equation}\label{eq_min}
A^{n+1}A^-+\ga_{n-1}A^nA^-+\dots +\ga_1 A^2A^-+\ga_0 I + (I-AA^-)+ \sum_{j=1}^{n-1} \ga_j (I-AA^-) =0.% \ga_{n-1}(I-AA^-)+\dots +\ga_1(I-AA^-) =0.
\end{equation}
Given $A^=\in A\{1\}$, and multiplying (\ref{eq_min}) on the right hand side by $AA^=$, we have the equality
\begin{equation}\label{eq_min2}
A^{n+1}A^=+\ga_{n-1} A^nA^=+\dots +\ga_1A^2A^=+\ga_0 AA^= = 0,
\end{equation}
and in particular $A^{n+1}A^-+\ga_{n-1}A^nA^-+\dots +\ga_1A^2A^-+\ga_0 AA^- =0$. The latter together with (\ref{eq_min}) imply that 
\begin{equation}\label{eq_min3}
(I-AA^-) + \ga_{n-1}(I-AA^-) + \dots + \ga_1(I-AA^-) + \ga_0(I-AA^-) =0.
\end{equation}
Again using \cite[Corollary 1, p.52]{Adi}, we know there exists $Z$ such that $I-AA^==(I-AZ)(I-AA^-)$. Multiplying the equality (\ref{eq_min3}) by $I-AZ$ on the left hand side, we obtain
\begin{equation}\label{eq_min4}
(I-AA^=) + \ga_{n-1}(I-AA^=) + \dots + \ga_1(I-AA^=)+\ga_0(I-AA^=)=0.
\end{equation}
Adding (\ref{eq_min2}) to (\ref{eq_min4}) we conclude $\psi_{A^2A^-+I-AA^-}(\lb)$ is an annihilating polynomial of $A^2A^=+I-AA^=$, and therefore $\psi_{A^2A^=+I-AA^=}(\lb)\mid \psi_{A^2A^-+I-AA}(\lb)$. Reversing the roles of $A^-$ and $A^=$, we obtain $\psi_{A^2A^=+I-AA^=}(\lb)= \psi_{A^2A^-+I-AA}(\lb)$.

So, for every choice of $A^-\in A\{1\}$, we have $$\psi_{A^2A^-+I-AA^-}(\lb)=  \lcm\{\lb^{-k}\psi_A(\lb), \lb^{k-1}(\lb-1)\} =  \lcm\{\lb\io\psi_A(\lb), \lb-1\}.$$  \endproof

As an example, consider $A=\bmx C & 0\\ 0 & N\emx = \left[\begin{array}{rr|rrr}
-2 & 1 & 0 & 0 & 0 \\
0 & -2 & 0 & 0 & 0 \\
\hline
 0 & 0 & 0 & 1 & 0 \\
0 & 0 & 0 & 0 & 1 \\
0 & 0 & 0 & 0 & 0
\end{array}\right] $ with $A^- = \bmx C\io & 0\\ 0 & N^T\emx$, which gives $ A^2A^- +I-AA^- = \left[\begin{array}{rr|rrr}
-2 & 1 & 0 & 0 & 0 \\
0 & -2 & 0 & 0 & 0 \\ \hline
0 & 0 & 0 & 1 & 0 \\
0 & 0 & 0 & 0 & 0 \\
0 & 0 & 0 & 0 & 1
\end{array}\right]$. We obtain $\psi_A(\lb) = (\lb + 2)^{2} \cdot \lb^{3} $ and $\psi_{A^2A^-+I-AA^-}(\lb) = (\lb - 1) \cdot \lb^{2} \cdot (\lb + 2)^{2}$.

\section{The index of an anti-triangular matrix}

In this section, we draw our attention to the Drazin index (and the expression of the Drazin inverse) of a block matrix of the form $  M=\mxl{cc} A & B \\  C & 0 \mxr $, where $A\in \F^{n\xx n}, B\in \F^{n\xx m},  C\in \F^{m\xx n}$ and the zero block is $m\xx m$. To the authors' knowledge, a general formula for the Drazin inverse of such a block matrix is not known, let alone tighter bounds for its index. We will use constraints on the blocks  in order to obtain  tractable  bounds on the index of $M$ related to the indices of its blocks.

\iffalse 
We firstly revisit  a special case concerning   group invertibility \cite[Corollary 2.2]{Group220}.

\begin{prop}
The block matrix $\mxl{cc} A & I \\  C & 0 \mxr $ is group invertible if and only if $C-A(I-C^-C)$ is a nonsingular matrix, for one and hence all choices of $C^-\in C\{1\}$.
\end {prop}

\fi

{ 
\begin{thm}\label{MainTh}
Let $M=\mxl{cc} A & B \\  C & 0 \mxr$  be a singular block matrix, and assume that both $A$ and $BC$ are singular matrices. %  where $A$ and $BC$ are singular square matrices over a field.

\ben 

\item $M$ is group invertible if and only if $$A(I-C^-C)-BC+(I-ZZ^-)(I-BB^-)(I+AC^-C -C^-C)$$ is nonsingular, for one and hence all choices of $B^-\in B\{1\}$, $C^-\in C\{1\}$,  $Z = (I-BB^-)A(I-C^-C)$, $Z^-\in Z\{1\}$.

\item $i(M)=2$ if and only if $i(M)>1$ %$M$ is not group invertible
 and $BC-\left(I-BC(BC)^-\right)A$ is nonsingular, for one and hence all choices of  $(BC)^-\in (BC)\{1\}$.
 
 \een

%Otherwise,
For $i(M)>2$:

\ben

\item[3.] If $ABC=BCA=0$ then  $$     M^D = 
    \mxl{cc}
    A^D & (A^D)^2 B + B (CB)^D \\
    C (A^D)^2 + (CB)^D C & C (A^D)^3 B
    \mxr.
$$

Moreover,
\ben
\item If $i(A)\ne 2 i(BC)-1$ or $i(A)=1$ or $i(BC)=1$ then
$$
    \max\{i(A),\, 2i(BC) - 1\} \;\le\; i(M) \;\le\; \max\{i(A),\, 2i(BC) - 1\} + 2.$$ 
\item If $i(A)>1$ and $i(BC)>1$ and $i(A)=2i(BC)-1$ then
$$
 i(A) - 1  \;\le\; i(M) \;\le\;  i(A)  + 2.$$ 
\een

\item[4.]  If   $ABC=0$ then 
$$\max\{i(A), 2i(BC)-1\} -1 \leq i(M) \leq i(A)+2i(BC)+2$$ and
 $$    M^D = \mxl{cc} A & I \\  C & 0 \mxr
    \mxl{cc} G_1 & G_2 \\  G_3 & G_4 \mxr^2
    \mxl{cc} I & 0 \\  0 & B \mxr,
$$
where 
\begin{eqnarray*}
W &=& BC\\
\ell & = & \max\{i(A),\,2i(W)-1\}-1\\
G_1 &=& (I-WW^D)(I+\alpha)A^D+W^D(I+\gamma)(I-AA^D)A \\
G_2 &=& (I-WW^D)(I+\alpha)(A^D)^2+W^D(I+\gamma)(I-AA^D)\\
G_3 &=& (I-WW^D)\beta A^D+W^D\delta (I-AA^D)A-WW^DAA^D +WW^D\\
G_4 &=& (I-WW^D)\beta (A^D)^2+W^D\delta (I-AA^D)-WW^DA^D\\
\alpha &=& \sum_{\substack{1\leq n\leq \ell\\ n\ \mathrm{even}}}
W^{\frac{n}{2}}(A^{D})^{n}\\
\beta &=&\sum_{\substack{1\leq n\leq \ell\\ n\ \mathrm{odd}}}
W^{\frac{n+1}{2}}(A^{D})^{n}\\
\gamma &=& \sum_{\substack{1\leq n\leq \ell\\ n\ \mathrm{even}}}
(W^{D})^{\frac{n}{2}}A^{n}\\
\delta &=& \sum_{\substack{1\leq n\leq \ell\\ n\ \mathrm{odd}}}
(W^{D})^{\frac{n+1}{2}}A^{n},
\end{eqnarray*}
where, as usual, a sum is understood to be zero whenever its index set is
empty. More generally, the same formulas hold if $\ell$ is replaced
by any integer satisfying $\ell\geq \max\{i(A),\,2i(W)-1\}-1$.
\een

\end{thm}
\proof

The equivalence (1) was proved in \cite[Theorem 2.1]{Group220}.

Before we address the remaining items of the theorem, we start by considering the factorization 
\begin{equation} \label{Dls}
    M=\mxl{cc} A & B \\   C & 0 \mxr = \mxl{cc} A & I \\   C & 0 \mxr
    \mxl{cc} I & 0 \\   0 & B \mxr = SR,
\end{equation} and we assume both $BC$ and $A$ to be singular.  

Let $\Gamma =RS=\mxl{cc} A & I \\   BC & 0 \mxr=\mxl{cc} A & I \\  W & 0 \mxr$, with $W=BC$. Applying Lemma \ref{Cline}, 
\begin{equation}
\label{MDm}
    M^D=S((\Gamma)^D)^2R  
\end{equation} from which 
\begin{equation*}
    M^D = \mxl{cc} A & I \\   C & 0 \mxr
    \left(\mxl{cc} A & I \\  BC & 0 \mxr^D\right)^2
    \mxl{cc} I & 0 \\  0 & B \mxr \, \text{ and } \, |i(M)-i(\Gamma
)|\leq 1.
\end{equation*}

Since $BC$ is singular, then $\Gamma$ is singular. Using  Theorem  \ref{DrazinInv} we obtain  
$i(\Gamma)=i(\Omega)+1$, where $\Omega=\Gamma^2\Gamma^{-}+I- \Gamma \Gamma^{-}$, with
$\Gamma^D=[\Omega^D]^2\Gamma$.

We now write  
$$\Gamma=\mxl{cc} A & I \\  W & 0 \mxr = \mxl{cc} I & A \\  0 & W \mxr
\mxl{cc} 0 & I \\  I & 0 \mxr = TP$$ and factor 
$$T=\mxl{cc} I & A \\  0 & W \mxr=\mxl{cc} I & 0 \\  0 & W \mxr
\mxl{cc} I & A \\  0 & I \mxr=DQ,$$ which gives 
$$ \Gamma \Gamma^{-}=DQPP^{-1}Q^{-1}D^{-}=DD^{-}=\mxl{cc} I & 0 \\  0 & WW^{-} \mxr.$$
We therefore obtain 
$$
\Omega =\mxl{cc} A & I \\  W & 0 \mxr
\mxl{cc} I & 0 \\  0 & WW^{-} \mxr +
\mxl{cc} 0 & 0 \\  0 & I-WW^{-} \mxr = \mxl{cc} A & WW^{-} \\  W & I-WW^{-} \mxr.
$$ 
Applying Theorem \ref{DrazinInv}, we know $\Gamma$ is group invertible precisely when $\Omega$ is nonsingular. Since $\bmx I & I\\ 0 & I\emx \Omega \bmx 0 & I\\ I & 0\emx = \bmx I & A+W\\ I-WW^- & W\emx$, this occurs exactly when $BC-\left(I-BC(BC)^-\right)A$ is nonsingular, from Lemma  \ref{schurcomp}. So, and since $i(M)>1$ and $i(\Gamma)=1$% $BC$ is singular, and therefore $i(M)\ne0$, and $M$ is not group invertible
, we necessarily have $i(M)=2$.

\

We now address the remaining cases of the theorem in which $i(M)>2$, and therefore assume $\Omega$ is singular.

Note that $\Omega =    \mxl{cc} A & WW^{-} \\  W & 0 \mxr + \mxl{cc} 0 & 0 \\  0 & I-WW^{-} \mxr$ is an orthogonal sum. Then
$$
\Omega^D  =\mxl{cc} A & WW^{-} \\  W & 0 \mxr^D + \mxl{cc} 0 & 0 \\  0 & I-WW^{-} \mxr,
$$ as $\mxl{cc} 0 & 0 \\  0 & I-WW^{-} \mxr$ is idempotent. Since 
$i(\Omega)=\max\left\{ i\left( \mxl{cc} A & WW^{-} \\  W & 0 \mxr \right), i\left( \mxl{cc} 0 & 0 \\  0 & I-WW^{-}\mxr\right) \right\}$ and  $W$ is not invertible then  $i(\Omega)= i\left(\mxl{cc} A & WW^{-} \\  W & 0 \mxr\right)$. 

Concerning the index of $\Omega$,     note that  $$ \mxl{cc} A & WW^{-} \\  W & 0 \mxr=\mxl{cc} A & 0 \\  0 & 0 \mxr + \mxl{cc} 0 & WW^{-} \\  W & 0 \mxr=X+Y.$$ 
 
Since $A$ is singular then $i(X)=i(A)$ with $ X^D=\bmx A^D & 0\\ 0 & 0\emx$. Note that if   $A$ were nonsingular then $i(X)=1=i(A)+1$ and $X^\# = \bmx A^{-1} & 0\\ 0 & 0\emx$.
  Moreover $Y^D=\mxl{cc} 0 & W^DWW^{-} \\ WW^D & 0 \mxr$ and $i(Y)=2i(W)-1$, from Lemma \ref{Y_from_W}.  

\

We are left to examine statements 3. and 4. of the theorem.

\ben 
\item[3.]
The equality $ABC=0=BCA$ is equivalent to  $XY=0=YX$ as 
$XY=\mxl{cc} 0 & AWW^{-} \\  0 & 0 \mxr$ and $YX=\mxl{cc}0 & 0 \\ WA & 0 \mxr$.
 
 In this case, $\Omega^D = (X+Y)^D = X^D+ Y^D$, that is,  $$\mxl{cc} A & WW^{-} \\  W & 0 \mxr^D=\mxl{cc} A^D & 0 \\   0 & 0 \mxr + \mxl{cc} 0 & W^DWW^{-} \\ WW^D & 0 \mxr . $$

 \ben
 \item[(a)] If $i(A)\ne 2 i(BC)-1$ or $i(A)=1$ or $i(BC)=1$ then $$i(\Omega) = i(X+Y)=\max\{ i(A), 2i(W)-1\}$$ by Lemma \ref{orthsum}(1). Therefore,   $$i(\Gamma)=i(\Omega)+1=i(X+Y)+1 = \max\{i(A), 2i(W)-1\}+1,$$ and   since  $|i(M)-i(\Gamma)|\leq 1$ we obtain $$ \max\{i(A),2i(W)-1\}\leq i(M) \leq \max\{i(A), 2i(W)-1\}+2.$$ 
    
\item[(b)]  If $i(A)>1$ and $i(BC)>1$ and $i(A)=2i(BC)-1$ then $$i(A) -1 \le i(\Omega) = i(X+Y)\le   i(A)$$ by Lemma \ref{orthsum}(2). Therefore, %  $i(\Gamma)=i(\Omega)+1=i(X+Y)+1 = \max\{i(A), 2i(W)-1\}+1$, and   since  $|i(M)-i(\Gamma)|\leq 1$ we obtain $ \max\{i(A),2i(W)-1\}\leq i(M) \leq \max\{i(A), 2i(W)-1\}+2$. 
$i(A)\le i(\Gamma) \le i(A)+1$ which imply $$  i(A)-1\le i(M)\le i(A)+2.$$
%$$  i(A) - 1  \;\le\; i(M) \;\le\;  i(A)  + 2.$$ 
   \een
 
Let us now proceed to compute the Drazin inverse of $M$. 
From Lemma \ref{Y_from_W},   $Y^D=\mxl{cc} 0 & W^DWW^{-} \\ WW^D & 0 \mxr$. Also, $X^D=\mxl{cc} A & 0 \\ 0 & 0 \mxr^D=\mxl{cc} A^D& 0 \\ 0 & 0 \mxr.$ Therefore, $$\Omega^D=\mxl{cc} A^D & 0 \\ 0 & 0 \mxr+\mxl{cc} 0 & W^DWW^{-} \\ WW^D & 0\mxr+\mxl{cc} 0 & 0 \\ 0 & I-WW^{-}\mxr = \mxl{cc} A^D & W^DWW^{-} \\ WW^D & I-WW^{-}\mxr.$$

From Theorem \ref{DrazinInv}, $\Gamma^D=(\Omega^D)^2\Gamma$, that is,

\begin{eqnarray*}
    \Gamma^D &=& \mxl{cc} A^D & W^DWW^{-} \\ WW^D & I-WW^{-}\mxr \mxl{cc} A^D & W^DWW^{-} \\ WW^D & I-WW^{-}\mxr \mxl{cc} A & I \\ W & 0 \mxr \\
    &=& \mxl{cc} (A^D)^2 + W^DWW^{-}WW^D & A^DW^DWW^{-}+W^DWW^{-}-W^DWW^{-}WW^{-} \\ 
    WW^DA^D+WW^D-WW^{-}WW^D & WW^DW^DW^{-}+I-WW^{-}\mxr \times \\
    &\times& \mxl{cc} A & I \\ W & 0 \mxr \\
    &=& \mxl{cc} (A^D)^2 + W^D & 0 \\ 0 & W^DWW^{-}+I-WW^{-}\mxr \mxl{cc} A & I \\ W & 0 \mxr \\
    &= &\mxl{cc} (A^D)^2A + W^DA & (A^D)^2 + W^D \\ W^DWW^{-}W+W-WW^{-}W & 0 \mxr \\
    &=& \mxl{cc} A^DAA^D + W^DA & (A^D)^2 + W^D \\ W^DW & 0 \mxr \\
    &=& \mxl{cc} A^D & (A^D)^2 + W^D \\ W^DW & 0 \mxr.
\end{eqnarray*} In the above, we use the fact that $AW=0$ implies $(A^D)^2AW(W^D)^2=0,$  which in turn means $  A^DAA^DW^DWW^D=0$; that is,  $A^DW^D=0.$ Subsequently,

\begin{align*}
    (\Gamma^D)^2 &= \mxl{cc} A^D  & (A^D)^2 + W^D \\ W^DW & 0 \mxr\mxl{cc} A^D & (A^D)^2 + W^D \\ W^DW & 0 \mxr \\
    &= \mxl{cc} (A^D)^2+(A^D)^2W^DW+(W^D)^2W  & (A^D)^3 + A^DW^D \\ W^DWA^D & W^DW(A^D)^2+W^DWW^D \mxr \\
    &= \mxl{cc} (A^D)^2+W^D  & (A^D)^3 \\ 0 & W^D \mxr. 
\end{align*}
$\textrm{since} \quad A^DW^D=0 \quad \textrm{and} \quad WA=0$ which implies that $WA(A^D)^2=0(A^D)^2, \quad \textrm{and} \quad WA^D=0.$ 

In order to compute $M^D=\mxl{cc} A & I \\  C & 0 \mxr (\Gamma^D)^2 \mxl{cc} I & 0 \\ 0 & B \mxr,$ we have

\begin{align*}
    M^D &=\mxl{cc} A & I \\ C & 0 \mxr \mxl{cc} (A^D)^2+W^D  & (A^D)^3 \\ 0 & W^D \mxr \mxl{cc} I & 0 \\ 0 & B \mxr \\
    &= \mxl{cc} A(A^D)^2+AW^D  & A(A^D)^3+W^D \\ C(A^D)^2+CW^D & C(A^D)^3 \mxr \mxl{cc} I & 0 \\ 0 & B \mxr \\
    &= \mxl{cc} A^D  & (A^D)^2B+B(CB)^D \\ C(A^D)^2+(CB)^DC & C(A^D)^3B \mxr.
\end{align*}

Note that    $W^D = (BC)^D = B\left((CB)^D\right)^2C$ by applying Lemma \ref{Cline}. This implies $W^D B = B\left((CB)^D\right)^2 CB = B(CB)^D (CB)^D CB = B(CB)^D,$ 
and  $CW^D = CB\left((CB)^D\right)^2 C = CB(CB)^D (CB)^D C = (CB)^DC.$

\item[4.] We now turn to the case $AW=ABC=0$, or equivalently, $XY=0$.

Applying  Lemma \ref{Cline} to   $X+Y=\bmx Y & I \emx \mxl{c} I \\ X \mxr,$ we obtain  $$(X+Y)^D=\bmx Y & I \emx \left(\mxl{cc} Y & I \\ 0 & X \mxr^D \right)^2 \mxl{c} I \\ X \mxr$$ with  $ |i(X+Y)-i  \left(\mxl{cc} Y & I \\ 0 & X \mxr \right) | \leq 1$. 
   Since $i(X+Y)=i(\Omega)= i(\Gamma)-1$, the following sequence of implications hold:
   \beqn
& &      \max\{i(X), i(Y)\}  \leq  i  \left(\mxl{cc} Y & I \\ 0 & X \mxr \right) \leq i(X)+i(Y)   \\    &\Rightarrow&  \max\{i(A), 2i(W)-1\}  \leq i  \left(\mxl{cc} Y & I \\ 0 & X \mxr \right) \leq i(A)+2i(W)-1 \\
    & \Rightarrow & \max\{i(A), 2i(W)-1\}-1  \leq i  (X+Y) \leq i(A)+2i(W) \\
    & \Rightarrow  & \max\{i(A), 2i(W)-1\}  \leq  i(\Gamma)\leq i(A)+2i(W)+1 \\
    & \Rightarrow & \max\{i(A), 2i(W)-1\} -1 \leq i(M) \leq i(A)+2i(W)+2.   
\eeqn

The expression for $M^D$ can be obtained via $\Gamma^D$, which in turn can be obtained via $\Omega^D= 
(\Gamma^2\Gamma^{-})^D + I - \Gamma\Gamma^{-} $. Since   $\Gamma^2\Gamma^{-} = X + Y $, we need to compute $(X+Y)^D$.

By \cite[Theorem~2.1]{Hartwig2001}, there exists an integer $k$
satisfying
$$\max\{i(X),i(Y)\}\leq k\leq i(X)+i(Y)$$
such that
\[
\begin{split}
(X+Y)^D
={}&
(I-YY^D)
\left(
I+\sum_{n=1}^{k-1}Y^n(X^D)^n
\right)X^D
\\
&+
Y^D
\left(
I+\sum_{n=1}^{k-1}(Y^D)^nX^n
\right)
(I-XX^D).
\end{split}
\]

Set
$$\ell=\max\{i(X),i(Y)\}-1.$$
We now show that the upper limit $k-1$ in the preceding expression
may be replaced by $\ell$.

Indeed, from the defining properties of the Drazin inverse, for every
integer $n\geq i(Y)$, we have
$$
Y^{n+1}Y^D=Y^n.
$$
Since $Y$ and $Y^D$ commute, it follows that
$$(I-YY^D)Y^n=Y^n-Y^{n+1}Y^D=0\qquad\text{for every }n\geq i(Y).$$
Similarly,
$$X^n(I-XX^D)=X^n-X^{n+1}X^D=0\qquad\text{for every }n\geq i(X).$$
Since $n\geq\ell+1$ implies $n\geq\max\{i(X),i(Y)\}$, we obtain
$$(I-YY^D)\sum_{n=\ell+1}^{k-1}Y^n(X^D)^n=0$$
and
$$Y^D\sum_{n=\ell+1}^{k-1}(Y^D)^nX^n(I-XX^D)=0.$$
Consequently,
\begin{equation}
\label{hartwig}
\begin{split}
(X+Y)^D
={}&
(I-YY^D)
\left(
I+\sum_{n=1}^{\ell}Y^n(X^D)^n
\right)X^D
\\
&+
Y^D
\left(
I+\sum_{n=1}^{\ell}(Y^D)^nX^n
\right)
(I-XX^D).
\end{split}
\end{equation}
Here and throughout the proof, a sum is understood to be zero whenever
its index set is empty.

Since $i(X)=i(A)$ and $i(Y)=2i(W)-1$, 
$$\ell=\max\{i(A),2i(W)-1\}-1.$$

Note that $\Omega = (X+Y)+E$ with $E=\mxl{cc} 0 & 0\\ 0 & I-WW^-\mxr$, and $(X+Y)E=E(X+Y)=0$, which implies $\Omega^D=(X+Y)^D + E$ and $\left(\Omega^D\right)^2 = \left( \left( X+Y\right)^D\right)^2 + E$. This will allow to obtain $\Gamma^D=(\Omega^D)^2 \Gamma = \left( (X+Y)^D\right)^2\Gamma$, since $E\Gamma=0$. 

Since $XY=0$ then clearly $X^DY=XY^D=X^D Y^D=0$ and $(I-XX^D)(I-YY^D)=I-XX^D-YY^D$. Squaring \eqref{hartwig}, we therefore obtain
\begin{eqnarray}
\label{squareDinv}
\left((X+Y)^D\right)^2
&=&
(I-YY^D)
\left(
I+\sum_{n=1}^{\ell}Y^n(X^D)^n
\right)
(X^D)^2
\nonumber\\
&&+
(Y^D)^2
\left(
I+\sum_{n=1}^{\ell}(Y^D)^nX^n
\right)
(I-XX^D)
-Y^DX^D.
\end{eqnarray}
 
Moreover,
\begin{eqnarray*}
I+\sum_{n=1}^{\ell}Y^n(X^D)^n
&=&
I+
\sum_{\substack{1\leq n\leq\ell\\ n\ {\rm odd}}}
Y^n(X^D)^n
+
\sum_{\substack{1\leq n\leq\ell\\ n\ {\rm even}}}
Y^n(X^D)^n
\\
&=&
I+
\sum_{\substack{1\leq n\leq\ell\\ n\ {\rm odd}}}
\bmx
0&0\\
W^{\frac{n+1}{2}}(A^D)^n&0
\emx
\\
&&+
\sum_{\substack{1\leq n\leq\ell\\ n\ {\rm even}}}
\bmx
W^{\frac{n}{2}}(A^D)^n&0\\
0&0
\emx
\\
&=&
\bmx
I+\alpha&0\\
\beta&I
\emx,
\end{eqnarray*}
where
\[
\alpha
=
\sum_{\substack{1\leq n\leq\ell\\ n\ {\rm even}}}
W^{\frac{n}{2}}(A^D)^n,
\qquad
\beta
=
\sum_{\substack{1\leq n\leq\ell\\ n\ {\rm odd}}}
W^{\frac{n+1}{2}}(A^D)^n.
\]

Likewise,
\begin{eqnarray*}
I+\sum_{n=1}^{\ell}(Y^D)^nX^n
&=&
I+
\sum_{\substack{1\leq n\leq\ell\\ n\ {\rm odd}}}
(Y^D)^nX^n
+
\sum_{\substack{1\leq n\leq\ell\\ n\ {\rm even}}}
(Y^D)^nX^n
\\
&=&
I+
\sum_{\substack{1\leq n\leq\ell\\ n\ {\rm odd}}}
\bmx
0&0\\
(W^D)^{\frac{n+1}{2}}A^n&0
\emx
\\
&&+
\sum_{\substack{1\leq n\leq\ell\\ n\ {\rm even}}}
\bmx
(W^D)^{\frac{n}{2}}A^n&0\\
0&0
\emx
\\
&=&
\bmx
I+\gamma&0\\
\delta&I
\emx,
\end{eqnarray*}
where
\[
\gamma
=
\sum_{\substack{1\leq n\leq\ell\\ n\ {\rm even}}}
(W^D)^{\frac{n}{2}}A^n,
\qquad
\delta
=
\sum_{\substack{1\leq n\leq\ell\\ n\ {\rm odd}}}
(W^D)^{\frac{n+1}{2}}A^n.
\]

Recall that $X^D=\bmx A^D & 0´\\ 0 & 0\emx$, $Y^D=\bmx 0 & W^DWW^-\\ WW^D & 0\emx$, $(Y^D)^2 = \bmx W^D & 0\\ 0 & WW^DW^-\emx$, $I-YY^D=\bmx I-WW^D & 0\\ 0 & I-WW^DWW^-\emx$, and also $(I-WW^DWW^-)W=(I-WW^D)W$ and $WW^DW^-W^D=W^D$.

The first summand of (\ref{squareDinv}) is then $\bmx (I-WW^D)(I+\alpha)(A^D)^2 & 0\\ (I-WW^D)\beta (A^D)^2 & 0\emx$, whereas the second summand equals $\bmx W^D(I+\gamma) (I-AA^D) & 0\\W^D\delta (I-AA^D) & WW^D W^-\emx$. The third summand  is simply $-Y^DX^D = \bmx 0 & 0\\-WW^D A^D & 0\emx$. 

We now proceed to compute $\Gamma^D = \left( (X+Y)^D\right)^2 \Gamma$, which leads to
$$\Gamma^D=
\bmx G_1 &
G_2\\
G_3 &
G_4\emx$$
where 
\begin{eqnarray*}
G_1 &=& (I-WW^D)(I+\alpha)A^D+W^D(I+\gamma)(I-AA^D)A \\
G_2 &=& (I-WW^D)(I+\alpha)(A^D)^2+W^D(I+\gamma)(I-AA^D)\\
G_3 &=& (I-WW^D)\beta A^D+W^D\delta (I-AA^D)A-WW^DAA^D +WW^D\\
G_4 &=& (I-WW^D)\beta (A^D)^2+W^D\delta (I-AA^D)-WW^DA^D,
\end{eqnarray*}
leading to 
\begin{equation*}
    M^D = \mxl{cc} A & I \\  C & 0 \mxr
    \mxl{cc} G_1 & G_2 \\  G_3 & G_4 \mxr^2
    \mxl{cc} I & 0 \\  0 & B \mxr . 
\end{equation*}

\een

\endproof

The following example highlights the subtlety of Theorem \ref{MainTh}(3b) concerning the bounds for $i(M)$. To this end, consider the matrices $A=\left[\begin{array}{rrr}
0 & 1 & 0 \\
0 & 0 & -1 \\
0 & 0 & 0
\end{array}\right]$, $B=\left[\begin{array}{rrr}
0 & 1 & 0 \\
0 & 0 & 0 \\
0 & 0 & 0
\end{array}\right]$, $C=\left[\begin{array}{rrr}
0 & 0 & 0 \\
0 & 0 & 1 \\
0 & 0 & 0
\end{array}\right]$, and therefore $M=\left[\begin{array}{rrr|rrr}
0 & 1 & 0 & 0 & 1 & 0 \\
0 & 0 & -1 & 0 & 0 & 0 \\
0 & 0 & 0 & 0 & 0 & 0 \\
\hline
 0 & 0 & 0 & 0 & 0 & 0 \\
0 & 0 & 1 & 0 & 0 & 0 \\
0 & 0 & 0 & 0 & 0 & 0
\end{array}\right]$.  For this example, $i(A)=3=2i(BC)-1$,  and $ i(M)=2$.

\begin{cor}
Let $M=\mxl{cc} A & B \\  C & 0 \mxr $, where $A$ is singular and $BC = 0$. Then
$$
   i(A)\le i(M) \le i(A) + 2.
$$ with $$M^D=\mxl{cc} A^D  & (A^D)^2B \\ C(A^D)^2 & C(A^D)^3B \mxr.$$
\end{cor}
\proof 
We first show that $M$ cannot be nonsingular. Suppose, to the contrary, that $M$ is invertible. Then $C$ must have full row rank and $B$ must have full column rank, which contradicts $BC=0$. Therefore, $M$ is singular.

Suppose now that $i(M)=1$. We shall prove that, in this case, $i(A)\leq 1$.

Let $v\in\ker A^2$. Thus, $A^2v=0$. Set
$$u=M
\mxl{c}
v\\
0
\mxr=\mxl{c}
Av\\
Cv
\mxr.
$$
Using $A^2v=0$ and $BC=0$, we obtain
$$
M^2u=\mxl{cc}
A & B\\
C & 0
\mxr
\mxl{c}
A^2v+BCv\\
CAv
\mxr=\mxl{c}
0\\
0
\mxr.
$$
Hence, $u\in\ker M^2$. Since $i(M)=1$, the kernels of $M$ and $M^2$ coincide. Therefore, $u\in\ker M$. On the other hand, by its definition, $u\in R(M)$. Since $M$ is group invertible, $
\ker M\cap R(M) =\{0\}$. Consequently, $u=0$. In particular, $Av=0$. Thus, every vector in $\ker A^2$ belongs to $\ker A$, and hence $\ker A^2=\ker A$. It follows that $i(A)\leq1$. Since $A$ is singular, its Drazin index cannot be zero. Therefore, $i(A)=1$ and then the inequalities $i(A)\le i(M) \le i(A) + 2$ hold. Moreover, the matrix appearing in the explicit formula for $M^D$ in the statement, expressed in terms of $A^D$, satisfies the three defining equations of the group inverse and therefore coincides with $M^\#$.

We note that from the singularity of $A$ one cannot have $i(M)=2$, applying Theorem \ref{MainTh}(2).

For $i(M)>2$ the result follows applying Theorem \ref{MainTh}(3(a)) since $i(BC)=1$. \endproof

We  present several examples that show that the inequalities in the previous Corollary are indeed the best possible. All matrices in the following examples are over the field $\mathbb{Q}$ of rational numbers.

In the following example, $i(A)=2$ and $i(M)=3$, with $$A=\left[\begin{array}{rrr}
\frac{1}{2} & 0 & 0 \\
0 & 2 & -2 \\
\frac{1}{2} & 2 & -2
\end{array}\right], B=\left[\begin{array}{rrrrrr}
1 & -\frac{1}{2} & -2 & -\frac{1}{2} & -1 & \frac{1}{2} \\
0 & -1 & 1 & 0 & 0 & -1 \\
1 & -\frac{3}{2} & -1 & -\frac{1}{2} & -1 & -\frac{1}{2}
\end{array}\right], C=\left[\begin{array}{rrr}
1 & 0 & 0 \\
0 & 1 & 0 \\
0 & 0 & 1 \\
0 & 0 & 0 \\
1 & -1 & -\frac{3}{2} \\
0 & -1 & 1
\end{array}\right]$$ and $ M=\bmx A& B\\ C & 0\emx.$ Also, $BC=0$ and  $A^D=\left[\begin{array}{rrr}
2 & 0 & 0 \\
-8 & 0 & 0 \\
-6 & 0 & 0
\end{array}\right]$, which gives \\ $M^D = \left[\begin{array}{rrr|rrrrrr}
2 & 0 & 0 & 4 & -2 & -8 & -2 & -4 & 2 \\
-8 & 0 & 0 & -16 & 8 & 32 & 8 & 16 & -8 \\
-6 & 0 & 0 & -12 & 6 & 24 & 6 & 12 & -6 \\
\hline
 4 & 0 & 0 & 8 & -4 & -16 & -4 & -8 & 4 \\
-16 & 0 & 0 & -32 & 16 & 64 & 16 & 32 & -16 \\
-12 & 0 & 0 & -24 & 12 & 48 & 12 & 24 & -12 \\
0 & 0 & 0 & 0 & 0 & 0 & 0 & 0 & 0 \\
38 & 0 & 0 & 76 & -38 & -152 & -38 & -76 & 38 \\
4 & 0 & 0 & 8 & -4 & -16 & -4 & -8 & 4
\end{array}\right]$.

In the next example, $i(A)=i(M)=3$. We take $$A= \left[\begin{array}{rrr}
0 & 1 & 0 \\
0 & 0 & 1 \\
0 & 0 & 0
\end{array}\right],  B= \left[\begin{array}{rrrrr}
3 & 0 & 0 & 0 & 0 \\
1 & 0 & 0 & 0 & 0 \\
0 & 0 & 0 & 0 & 0
\end{array}\right], C=\left[\begin{array}{rrr}
0 & 0 & 0 \\
0 & 1 & 4 \\
0 & 0 & 0 \\
0 & 0 & 0 \\
0 & 0 & 0
\end{array}\right]$$

\iffalse
In the next example, $i(A)=i(M)=1$. We take $$
A= \left[\begin{array}{rrrr}
0 & 1 & 0 & -\frac{1}{2} \\
-1 & 2 & 0 & -\frac{1}{2} \\
-1 & 3 & 0 & -1 \\
1 & -1 & -2 & -1
\end{array}\right], B= \left[\begin{array}{rrrrrr}
0 & 2 & 1 & 1 & 0 & 0 \\
1 & 0 & 0 & 1 & 0 & 1 \\
1 & 2 & 1 & 2 & 0 & 1 \\
0 & -2 & -1 & -1 & 0 & 0
\end{array}\right], C=0.$$
\fi

Finally, we present an example in which $i(M)=i(A)+2$. We take 
$$A= \left[\begin{array}{rrr}
1 & -1 & 1 \\
1 & 2 & 0 \\
2 & 1 & 1
\end{array}\right], B= \left[\begin{array}{rrrr}
-1 & 0 & -1 & -\frac{1}{2} \\
0 & 0 & 0 & 0 \\
0 & 0 & 0 & 0
\end{array}\right], C= \left[\begin{array}{rrr}
1 & 0 & 0 \\
0 & 0 & 0 \\
0 & 0 & 0 \\
-2 & 0 & 0
\end{array}\right]$$ in which $i(A)=1$.

\begin{cor}
Let  $M=\mxl{cc} A & B \\  C & 0 \mxr $ with $i(M)>2$, $i(BC)=1$ and $A$ singular. 
\ben
\item If $ABC=0=BCA$, then $ i(A)\le i(M)\le i(A)+2  $. In particular, if $i(A)=1$ then  $i(M)=3$.%$1\le i(M)\le 3$.
\item If $ABC=0$, then $ i(A)-1\le i(M)\le i(A)+4  $. In particular, if $i(A)=1$ then $i(M)\le 5$.
\een

\end{cor}

\begin{cor}
Given  $M=\mxl{cc} 0 & B \\  C & 0 \mxr $ with $BC$ singular, then
$$
  2i(BC)-1\le i(M) \le 2i(BC)+1.
$$ with $$M^D=\mxl{cc} 0  & B(CB)^D \\ (CB)^D C & 0 \mxr.$$
\end{cor}
 \proof We note that from the singularity of $BC$ one cannot have $i(M)=2$, applying Theorem \ref{MainTh}(2). Also,  $i(M)=1$ implies $i(BC)=1$ and the inequalities hold. We are left with the case $i(M)>2$.  The result follows applying Theorem \ref{MainTh}(3a). \endproof

We now consider the specific cases that we avoided in the previous result, namely $A$ being invertible and $BC$ being invertible. We note that in the case $A$ is nonsingular, then   $ABC=0$ is equivalent to $BC=0$.

\begin{thm}\label{th3-2}
Let $M=\mxl{cc} A & B \\  C & 0 \mxr $  where $A\in \F^{n\xx n}, B\in \F^{n\xx m}, C\in \F^{m\xx n}$, and suppose  further that $BC$ is nonsingular. Then $M$ is invertible if and only if  $n=m$,  $B$ and $C$ are invertible, and $i(M)=1$ otherwise. Furthermore, $$M^\# = \bmx 0 & (BC)\io B \\ C(BC)\io  & -C(BC)\io A (BC)\io B\emx.$$

\end{thm}

\proof %The first part of the result is trivial.

 Factoring   $M=\mxl{ll} A & B \\   C & 0 \mxr $ as $ \mxl{ll} A & I \\   C & 0 \mxr
    \mxl{ll} I & 0 \\   0 & B \mxr = SR,$ and since $i(RS)=0$ then either $M$ is nonsingular or $i(M)=1$. For the former, as $BC$ is nonsingular then $n\le m$. Trivially, if $m=n$ and $B,C$ are nonsingular, then $M$ is invertible. Conversely, if $\rank(M)=n+m$ then $n+m\le \rank(S)\le n+n$, which implies that $m\le n$. This means $m=n$ and $B,C$ are square matrices, and therefore $S$ and $R$ are invertible block matrices, leading to the invertibility of $B$ and $C$.

    For the expression of $M^\#$, we apply the formula $M^\# = S\left( (RS)\io\right)^2 R$ and the fact that $\bmx A & I\\ BC & 0\emx\io = \bmx 0 & (BC)\io \\ I & -A(BC)\io\emx$. \endproof 
    
 \begin{thm}
Let $M=\mxl{cc} A & B \\  C & 0 \mxr $  where $A$ and $BC$ are square matrices over a field. Suppose further that $A$ is nonsingular and $BC=0$. Then $i(M)=1$ if and only if $$A(I-C^+C)+(I-ZZ^-)(I-BB^+)(I+AC^+C-C^+C) $$  is invertible, where $Z=(I-BB^+)A(I-C^+ C)$, and for one choice, and hence all choices, of $B^+$, $C^+$ and $Z^-$. Otherwise, $i(M)=2$.

Moreover,
$$M^D = \bmx A\io & A^{-2}B\\
CA^{-2} & CA^{-3} B\emx.$$
 
\end{thm}   
\proof Consider the factorization   $M=\mxl{cc} A & B \\   C & 0 \mxr = \mxl{cc} A & I \\   C & 0 \mxr
    \mxl{cc} I & 0 \\   0 & B \mxr = SR$. Since $i(RS)=1$ with $(RS)^\#=\bmx A\io & A^{-2}\\ 0 & 0\emx$, then, and since $M$ cannot be invertible, either $i(M)=1$ or $i(M)=2$. For the former, we refer to \cite[Theorem 2.1]{Group220}.
    
The expression for $M^D$ follows from $(SR)^D = S\left((RS)^\#\right)^2R.$ \endproof

}

 \begin{thm}
 Let $M=\mxl{cc} A & I\\  C & 0 \mxr $  where $A$ and $C$ are square matrices over a field. The following hold:
 \ben
 \item $i(M)=0$ if and only if $i(C)=0$.
 \item If  $C$ is singular then $i(M)=1$ if and only if $i(C-A(I-C^-C))=0$ for one and hence all choices of $C^-\in C\{1\}$.
 \een
 
% Otherwise,  
For $i(M)\ge 2$:
\ben
 \item[3.] If   $AC=CA=0$ then  
 \ben 
 \item  if $i(A)\ne 2i(C)-1$ or $i(A)=1$ or $i(C)=1$ then $$i(M)=\max\{i(A)+1, 2i(C)\},$$
 \item  else %If $i(A)>1, i(C)>1$ and $i(A)=2i(C)-1$ then 
 $$i(A)\le i(M)\le i(A)+1.$$
 \een
 
 \item[4.] If $AC=0$ then $$\max\{i(A),2i(C)-1\}\le i(M)\le i(A)+2i(C)+1.$$
 \een
  \end{thm}
  \proof (1) is trivial and (2) follows from \cite[Corollary 2.2]{Group220}.
  
  For (3) and (4),  we remark that   $A$ is necessarily singular otherwise $C=0$ and hence $i(M)=1$, with $M^\# = \bmx A\io & A^{-2}\\ 0 & 0\emx$. We will use an analogous reasoning we took in the proof of Theorem \ref{MainTh}, by taking $W=C$ in $\Gamma$.  
   As such, there exists $M^-\in M\{1\}$ such that $MM^- = \bmx I & 0\\ 0 & CC^-\emx$, which leads to $\Omega =  M^2M^- +I-MM^-=\bmx A & CC^-\\ C & I-CC^-\emx = \bmx A & CC^-\\ C & 0\emx + \bmx 0 & 0\\ 0 & I-CC^-\emx$. This is an orthogonal sum  and hence, since $\Omega$ is singular as $i(M)\ge 2$, we obtain $i(\Omega) = i\left(\bmx A & CC^-\\ C & 0\emx\right)$. As in the proof of Theorem \ref{MainTh}, if $AC=0=CA$ we have,  and by considering the sum $\bmx A & CC^-\\ C & 0\emx = \bmx A & 0\\ 0 & 0\emx + \bmx 0& CC^-\\ C & 0\emx$ and its index related to the indices of the summands as in Lemma \ref{orthsum} and in Lemma \ref{Y_from_W},
    
   \ben
   \item[(a)] if $i(A)\ne 2i(C)-1$ or $i(A)=1$ or $i(C)=1$ then   $i(\Omega)=\max \{i(A),2i(C)-1\}$, and since $i(M)=i(\Omega)+1$, the result follows;
   \item[(b)] if $i(A)>1, i(C)>1, i(A)=2i(C)-1$ then $i(A)-1\le i(\Omega) \le i(A)$, from which $i(A)\le i(M)\le i(A)+1$.
   \een

  If $AC=0$ we repeat the steps of the proof of Theorem \ref{MainTh} in order to obtain $$\max\{i(A),2i(C)-1\}-1\le i(\Omega) \le i(A)+2i(C).$$ Since $i(M)=i(\Omega)+1$, the result follows. \endproof

\section{Applications to digraph matrices}

The intersection of generalized inverses and graph theory has garnered significant attention in academic literature due to the broad applicability of these subjects across diverse scientific domains. Key matrix representations, including the incidence matrix, adjacency matrix, and Laplacian matrix, are fundamental to the analysis of network flow, electrical networks, the definition of novel graph-theoretic distances, and the study of Markov processes. For a short introduction to this symbiosis, the reader is referred to \cite{Kelathaya}.

 Given a (weighted) digraph $D(A)= (V , E)$ with vertex set $V = \{1,\dots , n\}$ and  arc set $E\subseteq V\times V$, we construct the adjacency matrix $A$ by setting $a_{ij} = 1$ if and only if $ e = (i, j) \in E$. If we are in the presence of a weighted digraph, then there is a weight $w_{ij}\ne 0$ related to each arc that connects the vertex $v_i$ to the vertex $v_j$ , and in this case we consider the matrix $A=[w_{ij}]$. Note that if $A_1$ and $A_2$ are (weighted) adjacency matrices of the same graph then $A_1 = PA_2P\io$, for some permutation matrix $P$. The index of a matrix is invariant to matrix similarity, and if $A_1 = UAU\io$ then $A_1^D = UA^D U\io$. So, the considered order of the vertices   is irrelevant when addressing the index of these matrices.

 For example, any weighted bipartite digraph is fully characterized, up to permutation similarity, by an adjacency matrix of the form $\bmx 0 & B\\ C & 0\emx$,  where the zero blocks are square, called bipartite matrices. The group and Drazin inverses of these matrices were studied in  \cite{CatralDrazin, CatralGroup, CatralLAA}. We now apply Theorem \ref{MainTh}(1)  and Theorem \ref{th3-2} with $A=0$.
 \begin{thm} Given a bipartite matrix $ M=\bmx 0 & B\\ C & 0\emx$, then
 \ben
 \item if $BC$ is singular, then $ 2i(BC)-1\le i(M)\le 2i(BC)+1$ and  $$     M^D = 
    \mxl{cc}
    0 & B (CB)^D \\
    (CB)^D C & 0
    \mxr =  \mxl{cc}
    0 &   (B C)^D B \\
    C(BC)^D & 0
    \mxr.
$$
 \item if $BC$ is nonsingular, then $M$ is invertible if and only if $B$ and $C$ are invertible, and $M$ is group invertible otherwise. Moreover, 
 $$M^\# = \bmx 0 & (BC)\io B \\ C(BC)\io  & 0\emx.$$
 \een
 
 \end{thm}
 
 Note that \cite[Theorem 2.1]{CatralLAA} is a special case of (2) of the previous Theorem. Indeed, if $B=\bmx X & U\emx, C=\bmx Y\\V\emx$ with $\rank(UV)=1$, and $X,Y$ are invertible, then $UV= uv^*$, for some vectors $u,v$,  and $BC= XY+uv^*= X(I+X\io uv^*Y\io)Y$. The latter is invertible if and only if $I+X\io uv^*Y\io$ is invertible, which in turn is equivalent to $1+v^* (XY)\io u\ne 0$, using Sherman--Morrison--Woodbury formula, or by applying Theorem \ref{jacobson}.

 \iffalse
 \
 
 In \cite[Definition 2.1]{Mcdonald}, (real positive weighted)  linked stars  digraphs were considered. The adjacency matrix (up to permutation similarity) is of the form $\bmx A & B\\C & 0\emx$, with $B=\bmx \x_1^T & 0  & \cdots & 0\\ 0 & \x_2^T & & \vdots \\
 \vdots  & &\ddots& 0\\
 0 &\cdots &\cdots&  \x_n^T\emx$, $C = \bmx \y_1 & 0  & \cdots & 0\\ 0 & \y_2 & & \vdots \\
 \vdots  & &\ddots& 0\\
 0 &\cdots &\cdots&  \y_n\emx   $ and  strictly positive vectors $\x_i,\y_i \in \R^n$. Obviously $BC = \diag(\x_1^T y_1, \dots, \x_n^Ty_n)$ is a nonsingular matrix, both $B$ and $C$ are not invertible, and therefore $M$ is group invertible by   Theorem \ref{th3-2}.  A related case are double star digraphs, defined in  \cite[Definition 3.1]{Mcdonald}, whose (weighted) adjacency matrix is permutation similar to $\bmx A & B\\C & 0\emx$ with $A=\bmx 0 & a\\ b & 0\emx$, $B=\diag (\x^T, \z^T)$, $C=\diag( \y, \w)$, and $\x, \z \in \R^m$, $\y, \w \in \R^n$ strictly nonzero vectors. As in  \cite[Theorem 3.3]{Mcdonald}, assuming $\x^T \y\ne 0$ and $\z^T \w \ne 0$ means $BC$ is invertible, and  therefore $M$ is group invertible by   Theorem \ref{th3-2}.

 \fi
 
 \

\

\noindent \bf{Acknowledgement.} This research was partially financed by Portuguese Funds through FCT (Funda\c c\~ao para a Ci\^encia e a Tecnologia) within the Project UID/00013/2025. \\ \href{https://doi.org/10.54499/UID/00013/2025}{https://doi.org/10.54499/UID/00013/2025}

%We would like to sincerely thank all referees for their valuable comments and constructive suggestions, which have helped improve the quality and clarity of the manuscript.

%We would especially like to acknowledge one referee for its  thorough and constructive comments and suggestions. His/her thoughtful remarks and corrections were invaluable in strengthening   the scientific content of the original manuscript.

\end{document}